\documentclass{article}

\usepackage{amsmath,amssymb,amsthm}
\usepackage{a4wide}
\usepackage{cite}

\newtheorem{theorem}{Theorem}

\newtheorem{lemma}[theorem]{Lemma}
\newtheorem{proposition}[theorem]{Proposition}

\theoremstyle{definition}
\newtheorem{definition}[theorem]{Definition}
\theoremstyle{remark}
\newtheorem{remark}[theorem]{Remark}

\begin{document}

\author{Amar Debbouche$^{a,b}$\\
{\tt amar$_{-}$debbouche@yahoo.fr}
\and Delfim F. M. Torres$^{b}$\\
{\tt delfim@ua.pt}}

\title{Approximate Controllability of Fractional Delay Dynamic Inclusions 
with Nonlocal Control Conditions\thanks{This is a preprint of a paper whose 
final and definite form will be published in \emph{Applied Mathematics and Computation}, 
ISSN 0096-3003 (see {\tt http://www.sciencedirect.com/science/journal/00963003}). 
Submitted 12/Feb/2013; Revised 04/May/2014; Accepted 25/May/2014.}}

\date{$^{a}$Department of Mathematics, Guelma University, Guelma, Algeria\\[0.3cm]
$^{b}$CIDMA--Center for Research and Development in Mathematics and Applications,\\
Department of Mathematics, University of Aveiro, 3810-193 Aveiro, Portugal}

\maketitle


\begin{abstract}
We introduce a nonlocal control condition
and the notion of approximate controllability
for fractional order quasilinear control inclusions.
Approximate controllability of a fractional
control nonlocal delay quasilinear functional differential inclusion
in a Hilbert space is studied. The results are obtained by using the fractional
power of operators, multi-valued analysis,
and Sadovskii's fixed point theorem. Main result gives an
appropriate set of sufficient conditions for the considered system
to be approximately controllable. As an example, a fractional partial
nonlocal control functional differential inclusion is considered.

\bigskip


\smallskip

\end{abstract}


\section{Introduction}

We are concerned with the fractional delay quasilinear control inclusion
\begin{equation}
\label{eq:1.1}
D^{\alpha}_{t}[u(t)-g(t, u(\sigma(t)))]\in Au(t)
+\int_{0}^{t}f(t, s, B_{1}\mu_{1}(\delta(s)))ds
\end{equation}
subject to the nonlocal control condition
\begin{equation}
\label{eq:1.2}
u(0)+h(u(t))=B_{2}\mu_{2}(t)+u_{0},
\end{equation}
where the unknown $u(\cdot)$ takes its values in a Hilbert space $H$
with norm $\Vert\cdot\Vert$, $D^{\alpha}_{t}$ is the Caputo fractional
derivative with $0<\alpha\leq1$ and $t\in J=[0,a]$. Let $A$ be a closed linear
operator defined on a dense domain $D(A)$ in $H$ into $H$ that generates
an analytic semigroup $Q(t)$, $t\geq 0$, of bounded linear operators on $H$
and $u_{0}\in D(A)$. We assume that $\lbrace B_{i}: U\rightarrow H,~ i=1,2\rbrace$
is a family of bounded linear operators, the control functions $\mu_{i}$, $i=1,2$,
belong to the space $L^{2}(J, U)$, a Hilbert space of admissible control functions
with $U$ as Hilbert space, and $\sigma, \delta: J\rightarrow J^{\prime}$ are delay arguments,
$J^{\prime}=[0, t]$. It is also assumed that $g: J\times H\rightarrow H$
and $h: C(J^{\prime}: H)\rightarrow H$ are given abstract functions and
$f: \Delta\times H\rightarrow H$ is a multi-valued map,
$\Delta=\lbrace (t, s): 0\leq s\leq t\leq a\rbrace$.

Three centuries ago, fractional calculus (i.e., the calculus of non-integer
order derivatives and integrals) has been dealt almost by mathematicians only.
During the past decades, this subject and its potential applications have gained
a lot of importance, mainly because fractional calculus has become a powerful tool
with more accurate and successful results in modeling several complex phenomena
in numerous seemingly diverse and widespread fields of science and engineering
\cite{AMA.5,AMA.10,AMA.13,AMA.17,AMA.20,AMA.29,AMA.31,AMA.32}.
It was found that various, especially interdisciplinary applications, can be elegantly
modeled with the help of fractional derivatives. Several authors have demonstrated applications
in the frequency dependent damping behavior of viscoelastic materials \cite{AMA.3,AMA.4},
dynamics of interfaces between nanoparticles and substrates \cite{AMA.9},
the nonlinear oscillation of earthquakes \cite{AMA.19}, bioengineering \cite{AMA.27},
continuum and statistical mechanics \cite{AMA.28}, signal processing \cite{AMA.34},
filter design, robotics and circuit theory \cite{AMA.39}. Fractional differential
equations provide an excellent instrument for the description of memory
and hereditary properties of various materials and processes \cite{AMA.37}.

In control theory, one of the most important qualitative aspects of a dynamical control system is controllability.
The problem of controllability consists to show the existence of a control function that steers the solution
of the system from its initial state to a final state, where the initial and final states may vary over
the entire space. A large class of scientific and engineering problems is modeled by partial differential
equations, integral equations or coupled ordinary and partial differential, integrodifferential equations,
which arise in problems connected with heat-flow in materials with memory, viscoelasticity and many other
physical phenomena. So it becomes important to study controllability results of such systems using
available techniques. The concept of controllability plays a major role in finite-dimensional control theory,
so that it is natural to try to generalize it to infinite dimensions \cite{AMA.38}. Moreover, the exact
controllability for semilinear fractional order systems, when the nonlinear term is independent
of the control function, is proved by assuming that the controllability operator has an induced inverse
on a quotient space, see for example \cite{AMA.11,AMA.12}. However, if the semigroup associated with
the system is compact, then the controllability operator is also compact and hence the induced inverse
does not exist because the state space is infinite dimensional \cite{AMA.48}.
Thus, the concept of exact controllability is too strong and
has limited applicability, while approximate controllability
is a weaker concept completely adequate in applications.

In recent years, attention has been paid to establish sufficient conditions for the existence
and controllability of (fractional) differential equations and inclusions, see, for instance,
\cite{AMA.1,AMA.2,AMA.7,AMA.8,AMA.26,AMA.40,AMA.49}. Ntouyas and O'Regan \cite{AMA.33} studied
existence results for semilinear neutral functional differential inclusions, Fu \cite{AMA.18}
established approximate controllability for neutral nonlocal impulsive differential inclusions
and Yan \cite{AMA.50,AMA.51} investigated the question of approximate controllability
of both fractional neutral functional differential equations and fractional integro-differential inclusions
with state-dependent delays. For more works about the approximate controllability for fractional systems,
we refer the reader to \cite{MR3172421,AMA.41,AMA.42,AMA.43,AMA.44,AMA.45,AMA.46},
see also Kumar and Sukavanam \cite{AMA.24,AMA.47}. However, in the mentioned papers,
the control function is located only in the inhomogeneous part of the evolution system.
For this reason, and motivated by this fact, we construct here two control functions:
the first control depends on the multi-valued map and the other with the nonlocal condition.
In order to realize this new complex form, we introduce here the study of approximate controllability
for a class of fractional delay dynamic inclusions with nonlocal control conditions.

The paper is organized as follows. In Section~\ref{sec:2}, we review some essential facts from fractional calculus,
multi-valued analysis and semigroup theory, that are used to obtain our main results. In Section~\ref{sec:3},
we state and prove existence and approximate controllability results for the fractional control system
\eqref{eq:1.1}--\eqref{eq:1.2}. Finally, in Section~\ref{sec:4}, as an example, a fractional partial dynamical
differential inclusion with a nonlocal control condition is considered. We end with Section~\ref{sec:conc}
of conclusions and some possible future directions of research.


\section{Preliminaries}
\label{sec:2}

In this section we give some basic definitions, notations, propositions and lemmas,
which will be used throughout the work. In particular, we state main properties
of fractional calculus \cite{AMA.23,AMA.30,AMA.37}, elementary principles
of multi-valued analysis \cite{AMA.15,AMA.22}, and well known facts
in semigroup theory \cite{AMA.21,AMA.36,AMA.52}.

\begin{definition}
\label{Definition:2.1}
The fractional integral of order $\alpha>0$ of a function $f\in L^{1}([a,b],\mathbb{R}^{+})$ is given by
$$
I^{\alpha}_{a}f(t)=\frac{1}{\Gamma(\alpha)}\int_{a}^{t}(t-s)^{\alpha-1}f(s)ds,
$$
where $\Gamma$ is the gamma function. If $a=0$, we can write
$I^{\alpha}f(t) := I^{\alpha}_{0}f(t) = (g_{\alpha}*f)(t)$, where
\begin{equation}
\label{f:g}
g_{\alpha}(t):=\left\{
\begin{array}{ll}
\frac{1}{\Gamma(\alpha)}t^{\alpha-1},& \mbox{$t>0$},\\
0, & \mbox{$t\leq 0$},
\end{array}\right.
\end{equation}
and, as usual, $*$ denotes the convolution of functions.
\end{definition}

\begin{remark}
For function \eqref{f:g}, one has
$\lim\limits_{\alpha\rightarrow 0}g_{\alpha}(t)=\delta(t)$
with $\delta$ the delta Dirac function.
\end{remark}

\begin{definition}
\label{Definition:2.2}
The Riemann--Liouville fractional derivative of order
$n-1<\alpha<n$, $n\in \mathbb{N}$, for a
function $f\in C([0, \infty))$ is given by
$$
^{L}D^{\alpha}f(t)=\frac{1}{\Gamma(n-\alpha)}\frac{d^{n}}{dt^{n}}
\int_{0}^{t}\frac{f(s)}{(t-s)^{\alpha+1-n}}ds,
\quad t>0.
$$
\end{definition}

\begin{definition}
\label{Definition:2.3}
The Caputo fractional derivative of order $n-1<\alpha<n$, $n\in \mathbb{N}$,
for a function $f\in C^{n-1}([0,\infty))$ is given by
$$
^{C}D^{\alpha}f(t) = {^{L}D}^{\alpha}\left(f(t)
-\sum\limits_{k=0}^{n-1}\frac{t^{k}}{k!}f^{(k)}(0)\right),
\quad t>0.
$$
\end{definition}

\begin{remark}
\label{Remark:2.1}
The following properties hold (see, e.g., \cite{AMA.37}):
\begin{itemize}
\item[(i)] If $f\in C^{n}([0, \infty))$, then
$$
^{C}D^{\alpha}f(t)=\frac{1}{\Gamma(n-\alpha)}
\int_{0}^{t}\frac{f^{(n)}(s)}{(t-s)^{\alpha+1-n}}ds
=I^{n-\alpha}f^{n}(t), \quad t>0, \quad n-1<\alpha<n,
\quad n\in \mathbb{N}.
$$
\item[(ii)] The Caputo derivative of a constant is equal to zero.
\item[(iii)] If $f$ is an abstract function with values in $H$,
then the integrals that appear in
Definitions~\ref{Definition:2.1}--\ref{Definition:2.3}
are taken in Bochner's sense.
\end{itemize}
\end{remark}

According to previous definitions, it is suitable to rewrite problem
\eqref{eq:1.1}--\eqref{eq:1.2} as the equivalent integral inclusion
\begin{multline}
\label{eq:2.1}
u(t)\in B_{2}\mu_{2}(t)+u_{0}-h(u(t))-g(0, u(\sigma(0)))+g(t, u(\sigma_{1}(t)))\\
+\frac{1}{\Gamma(\alpha)}\int_{0}^{t}(t-s)^{\alpha-1}\left[Au(s)
+\int_{0}^{s}f(s, \eta, B_{1}\mu_{1}(\delta(\eta)))d\eta\right]ds,
\end{multline}
provided the integral in \eqref{eq:2.1} exists. Before formulating the definition
of mild solution to \eqref{eq:1.1}--\eqref{eq:1.2},
we first give the following notations, corollaries and lemmas.

Let $(X, \Vert\cdot\Vert)$ be a Banach space, $C(J, X)$ denote the Banach space
of continuous functions from $J$ into $X$ with the norm
$\Vert u\Vert_{J}=\sup \lbrace\Vert u(t)\Vert: t\in J\rbrace$, and let
$\mathcal{L}(X)$ be the Banach space of bounded linear operators from $X$ to $X$
with the norm $\Vert G\Vert_{\mathcal{L}(X)}
=\sup\lbrace\Vert G(u)\Vert: \Vert u\Vert=1\rbrace$. We also denote:
\begin{itemize}
\item $P(X)=\lbrace Y\in 2^{X}: Y \neq \emptyset\rbrace$,
\item $P_{cl}X=\lbrace Y\in P(X), Y$ is closed$\rbrace$,
\item $P_{b}X=\lbrace Y\in P(X), Y$ is bounded$\rbrace$,
\item $P_{c}X=\lbrace Y\in P(X), Y$ is convex$\rbrace$,
\item $P_{cp}X=\lbrace Y\in P(X), Y$ is compact$\rbrace$.
\end{itemize}
The following results on multi-valued analysis are useful.

\begin{proposition}[See \cite{AMA.15}]
\label{Proposition:2.1}
\begin{itemize}
\item[(1)] A measurable function $u: J\rightarrow X$ is Bochner integrable
if and only if $\Vert u\Vert$ is Lebesgue integrable.
\item[(2)] A multi-valued map $F: X\rightarrow 2^{X}$ is said to be convex-valued
(closed-valued) if $F(u)$ is convex (closed) for all $u\in X$; is said to be bounded
on bounded sets if $F(B)=\bigcup\limits_{u\in B}F(u)$ is bounded in $X$ for all $B\in P_{b}(X)$.
\item[(3)] A map $F$ is said to be upper semi-continuous (u.s.c.) on $X$ if for each
$u_{0}\in X$ the set $F(u_{0})$ is a nonempty closed subset of $X$, and if for each open
subset $\Omega$ of $X$ containing $F(u_{0})$, there exists an open neighborhood
$\bigtriangledown$ of $u_{0}$ such that $F(\bigtriangledown)\subseteq\Omega$.
\item [(4)] A map $F$ is said to be completely continuous if $F(B)$ is relatively compact
for every $B\in P_{b}(X)$. If the multi-valued map $F$ is completely continuous
with nonempty compact values, then $F$ is u.s.c. if and only if $F$ has a closed
graph, i.e., $u_{n}\rightarrow u, y_{n}\rightarrow y, y_{n}\in F(u_{n})$ imply $y\in F(u)$.
We say that $F$ has a fixed point if there is $u\in X$ such that $u\in F(u)$.
\item [(5)] A multi-valued map $F: J\rightarrow P_{cl}(X)$ is said to be measurable if for
each $u\in X$ the function $y: J\rightarrow \mathbb{R}$ defined by
$y(t)=d(u, F(t))=\inf \lbrace \Vert u-z\Vert, z\in F(t)\rbrace$ is measurable.
\item [(6)] A multi-valued map $F: X\rightarrow 2^{X}$ is said to be condensing if for any
bounded subset $B\subset X$ with $\beta(B)\neq 0$ we have $\beta(F(B))<\beta(B)$, where
$\beta(\cdot)$ denotes the Kuratowski measure of non-compactness defined as follows:
$\beta(B):=\inf \lbrace d>0: B$ can be covered by a finite number of balls of radius $d\rbrace$.
\end{itemize}
\end{proposition}

The following results are motivated by \cite{AMA.25}.

\begin{proposition}[Cf. \cite{AMA.25}]
\label{Proposition:2.2}
Let $X$ be a Banach space and let $f: \Delta\times X\rightarrow P_{b, cl, c}(X)$
satisfy the following conditions:
\begin{enumerate}
\item for each $\mu_{1}\in L^{2}(J^{\prime}, U)$,
$(t, s, B_{1}\mu_{1})\rightarrow f(t, s, B_{1}\mu_{1})$ is measurable on $\Delta$;
\item for each $(t, s)\in \Delta$, $(t, s, B_{1}\mu_{1})\rightarrow f(t, s, B_{1}\mu_{1})$
is u.s.c. with respect to $B_{1}\mu_{1}$;
\item for each fixed $\mu_{1}\in L^{2}(J^{\prime}, U)$, the set
$$
S_{f, \mu_{1}}=\left\lbrace v\in L^{1}(J^{\prime}, X):
B_{1}\mu_{1}(t)+v(\delta(t))\in \int_{0}^{t}f(t, s, B_{1}\mu_{1}(\delta(s)))ds,
\,  \text{ a.e. } t\in J\right\rbrace
$$
is nonempty.
\end{enumerate}
Also, let $P$ be a linear continuous mapping from $L^{1}(J^{\prime}, X)$ to $L^{2}(J^{\prime}, U)$.
Then the operator
\begin{equation*}
\begin{split}
P\circ S_{f, \mu_{1}}: L^{2}(J^{\prime}, U) &\longrightarrow P_{cp, c}(L^{2}(J^{\prime}, U))\\
\mu_{1} &\longmapsto P\circ S_{f}(\mu_{1}) =: P(S_{f, \mu_{1}})
\end{split}
\end{equation*}
is a closed graph operator.
\end{proposition}

\begin{lemma}[See \cite{AMA.15}]
\label{Lemma:2.3}
Let $\Omega$ be a bounded, convex, and closed set in the
Banach space $X$ and $F: \Omega\rightarrow 2^{\Omega}\setminus\lbrace \emptyset\rbrace$
be a u.s.c. condensing multi-valued map. If for every $u\in \Omega$, $F(u)$ is a closed
and convex set in $\Omega$, then $F$ has a fixed point in $\Omega$.
\end{lemma}

Throughout the paper, $(H, \Vert \cdot\Vert)$ is a separable Hilbert space.
If $A: D(A)\subset H\rightarrow H$ is the infinitesimal generator of a compact analytic
semigroup of uniformly bounded linear operators $Q(\cdot)$, then there exists a constant
$M\geq1$ such that $\Vert Q(t)\Vert\leq M$ for $t\geq0$. Without loss of generality,
we assume that $0\in \rho(A)$, the resolvent set $A$. Then it is possible to define
the fractional power $A^{q}$, for $0<q\leq1$, as a closed linear operator on its domain
$D(A^{q})$ with inverse $A^{-q}$. Furthermore, the subspace $D(A^{q})$ is dense in $H$
and the expression $\Vert u\Vert_{q}=\Vert A^{q}u\Vert, u\in D(A^{q})$ defines a norm on $D(A^{q})$.
Hereafter, we denote by $H_{q}$ the Banach space $D(A^{q})$ normed with $\Vert u\Vert_{q}$.

\begin{lemma}[See \cite{AMA.36}]
\label{Lemma:2.4}
Let $A$ be the infinitesimal generator of an analytic semigroup $Q(t)$.
If $0\in \rho(A)$, then
\begin{itemize}
\item [(a)] $Q(t): H\rightarrow D(A^{q})$ for every $t>0$ and $q\geq0$;
\item [(b)] $Q(t)A^{q}u=A^{q}Q(t)u$ for every $u\in D(A^{q})$;
\item [(c)] the operator $A^{q}Q(t)$ is bounded and
$\Vert A^{q}Q(t)\Vert\leq M_{q}t^{-q}e^{-\omega t}$ for every $t>0$;
\item [(d)] if $0<q\leq1$ and $u\in D(A^{q})$, then
$\Vert Q(t)u-u\Vert\leq C_{q}t^{q}\Vert A^{q}u\Vert$.
\end{itemize}
\end{lemma}

\begin{remark}
\label{Remark:2.2}
We note that:
\begin{itemize}
\item [(i)] $D(A^{q})$ is a Hilbert space with the norm
$\Vert u\Vert_{q}=\Vert A^{q}u\Vert$ for $u\in D(A^{q})$.
\item [(ii)] If $0<p\leq q\leq1$, then $D(A^{q})\hookrightarrow D(A^{p})$.
\item [(iii)] $A^{-q}$ is a bounded linear operator
in $H$ with $D(A^{q})=Im(A^{-q})$.
\end{itemize}
\end{remark}

Let us consider the set
$\Omega=\lbrace u: u\in C(J, H_{q}), q\in (0, 1)\rbrace$,
which is a Banach space with the norm
$\Vert u\Vert_{\Omega}=\sup_{t\in J}\Vert u(t)\Vert_{q}$.

\begin{definition}[Cf. \cite{AMA.14,AMA.16} and \cite{AMA.18,AMA.33,AMA.50,AMA.53}]
\label{Definition:2.4}
A state function $u(t)\in \Omega$ is called a mild solution of \eqref{eq:1.1}--\eqref{eq:1.2}
if $u(0)=B_{2}\mu_{2}(t)+u_{0}-h(u(t))$, the function
$(t-s)^{\alpha-1}AT_{\alpha}(t-s)g(s, u(\sigma(s)))$, $s\in J$, is integrable on $[0, t)$
for every $t\in J$, and for each control $\mu_{1}\in L^{2}(J, U)$ there exists a function
$v\in L^{1}(J^{\prime}, H)$ such that
$\displaystyle v(\delta(t))+B_{1}\mu_{1}(t)\in \int_{0}^{t}f(t, s, B_{1}\mu_{1}(\delta(s)))ds$
a.e. on $J$ and the following integral equation is satisfied:
\begin{multline}
\label{eq:2.2}
u(t)=S_{\alpha}(t)\left[B_{2}\mu_{2}(t)+u_{0}-h\left(u(t)\right)
-g\left(0, u(\sigma(0))\right)\right]+g\left(t, u(\sigma(t))\right)\\
+\int_{0}^{t}(t-s)^{\alpha-1}\left\lbrace AT_{\alpha}(t-s)g\left(s, u(\sigma(s))\right)
+T_{\alpha}(t-s)\left[v(\delta(s))+B_{1}\mu_{1}(s)\right]\right\rbrace ds,
\end{multline}
where
$$
S_{\alpha}(t)=\int_{0}^{\infty}\zeta_{\alpha}(\theta)Q(t^{\alpha}\theta)d\theta,
\quad T_{\alpha}(t)=\alpha\int_{0}^{\infty}\theta\zeta_{\alpha}(\theta)Q(t^{\alpha}\theta)d\theta,
$$
$$
\zeta_{\alpha}(\theta)=\frac{1}{\alpha}\theta^{-1-\frac{1}{\alpha}}
\varpi_{\alpha}(\theta^{-\frac{1}{\alpha}})\geq 0,
\quad \varpi_{\alpha}(\theta)
=\frac{1}{\pi}\sum_{n=1}^{\infty}(-1)^{n-1}\theta^{-\alpha n-1}
\frac{\Gamma(n\alpha+1)}{n!}\sin (n\pi\alpha), \theta\in (0, \infty),
$$
with $\zeta_{\alpha}$ the probability density function defined on $(0, \infty)$, that is,
$\zeta_{\alpha}(\theta)\geq 0$, $\theta\in(0, \infty)$, and
$\int_{0}^{\infty}\zeta_{\alpha}(\theta)d\theta=1$.
\end{definition}

\begin{lemma}[See \cite{AMA.50,AMA.53}]
\label{Lemma:2.5}
The operators $S_{\alpha}(t)$ and $T_{\alpha}(t)$ satisfy the following properties.
\begin{itemize}
\item[(a)] For any fixed $t\geq0$, $S_{\alpha}(t)$
and $T_{\alpha}(t)$ are linear and bounded operators, i.e.,
for any $u\in H$, $\Vert S_{\alpha}(t)u\Vert\leq M\Vert u\Vert$ and
$\Vert T_{\alpha}(t)u\Vert\leq \frac{M\alpha}{\Gamma (1+\alpha)}\Vert u\Vert$.
\item[(b)] $\lbrace S_{\alpha}(t), t\geq0\rbrace$ and $\lbrace T_{\alpha}(t), t\geq0\rbrace$
are strongly continuous, i.e., for $u\in H$ and $0\leq t_{1}<t_{2}\leq a$, we have
$\Vert S_{\alpha}(t_{2})u-S_{\alpha}(t_{1})u\Vert\rightarrow 0$ and
$\Vert T_{\alpha}(t_{2})u-T_{\alpha}(t_{1})u\Vert\rightarrow 0$ as $t_{1}\rightarrow t_{2}$.
\item[(c)] For every $t>0$, $S_{\alpha}(t)$ and $T_{\alpha}(t)$ are compact operators.
\item[(d)] For any $u\in H$, $p\in (0, 1)$ and $q\in (0, 1)$,
$AT_{\alpha}(t)u=A^{1-p}T_{\alpha}(t)A^{p}u$, $t\in J$, and
$$
\Vert A^{q}T_{\alpha}(t)\Vert \leq
\frac{\alpha M_{q}\Gamma(2-q)}{\Gamma(1+\alpha(1-q))}t^{-q\alpha},
\quad 0<t\leq a.
$$
\end{itemize}
\end{lemma}

Motivated by the recent works \cite{AMA.18,AMA.24,AMA.41,AMA.42,AMA.47,AMA.50},
we make use of the following notions.
Let $u_{a}(u(0); \mu_{1}, \mu_{2})$ be the state value of \eqref{eq:1.1}--\eqref{eq:1.2}
at terminal time $a$, corresponding to the controls $\mu_{1}$ and $\mu_{2}$
and the nonlocal control value $u(0)$. For every $u_{0}\in H$, we introduce the set
$$
\mathfrak{R}(a, u(0))=\left\lbrace
u_{a}\left(B_{1}\mu_{1}(t)+u_{0}-h(u(t));\mu_{1}, \mu_{2}\right)(0)
: \mu_{1}(\cdot),\mu_{2}(\cdot)\in L^{2}(J, U)\right\rbrace,
$$
which is called the \emph{reachable set} of system \eqref{eq:1.1}--\eqref{eq:1.2}
at terminal time $a$. Its closure in $H$ is denoted by $\overline{\mathfrak{R}(a, u(0))}$.

\begin{definition}
\label{Definition:2.5}
The system \eqref{eq:1.1}--\eqref{eq:1.2} is said to be approximately controllable on $J$
if $\overline{\mathfrak{R}(a, u(0))}=H$, that is, given an arbitrary $\epsilon>0$,
it is possible to steer from the point $u(0)$ at time $a$ all points in the state space
$H$ within a distance $\epsilon$.
\end{definition}

Consider the following linear nonlocal control fractional system:
\begin{equation}
\label{eq:2.3}
D^{\alpha}_{t}u(t)=Au(t)+B_{1}\mu_{1}(t),
\end{equation}
\begin{equation}
\label{eq:2.4}
u(0)=u_{0}+B_{2}\mu_{2}.
\end{equation}
The approximate controllability for the linear nonlocal control
fractional system \eqref{eq:2.3}--\eqref{eq:2.4} is a natural
generalization of the notion of approximate controllability
of a linear first-order control system ($\alpha=1$ and $B_{2}=0$).
It is convenient at this point to introduce the controllability operators
associated with \eqref{eq:2.3}--\eqref{eq:2.4} as
\begin{equation}
\label{eq:2.5}
\begin{aligned}
&\Gamma^{a}_{0,1}=\int_{0}^{a}(a-s)^{\alpha-1}
T_{\alpha}(a-s)B_{1}B_{1}^{\ast}T_{\alpha}^{\ast}(a-s)ds,\\
&\Gamma^{a}_{0,2}=S_{\alpha}(a)B_{2}B_{2}^{\ast}S_{\alpha}^{\ast}(a),
\end{aligned}
\end{equation}
where $S_{\alpha}^{\ast}(t)$, $T_{\alpha}^{\ast}(t)$ and $B_{j}^{\ast}$, $j=1, 2$,
denote the adjoints of $S_{\alpha}(t)$, $T_{\alpha}(t)$ and $B_{j}$, respectively.
Moreover, we give the relevant operators
\begin{equation}
\label{eq:2.6}
\mathcal{R}(\lambda, \Gamma^{a}_{0,i})
=\left(\lambda I+\Gamma^{a}_{0,i}\right)^{-1}
\end{equation}
for $i=1,2$ and $\lambda>0$. It is straightforward to see that
$\Gamma^{a}_{0,1}$ and $\Gamma^{a}_{0,2}$ are linear bounded operators.

\begin{lemma}[See \cite{AMA.6,AMA.50}]
\label{Lemma:2.6}
The fractional linear control system \eqref{eq:2.3}--\eqref{eq:2.4}
is approximately controllable on $J$ if and only if
$\lambda\mathcal{R}(\lambda, \Gamma^{a}_{0,i})\rightarrow 0$ as $\lambda\rightarrow 0^{+}$, $i=1,2$,
in the strong operator topology.
\end{lemma}


\section{Main Results}
\label{sec:3}

We obtain existence and approximate controllability results
for the fractional nonlocal control inclusion
\eqref{eq:1.1}--\eqref{eq:1.2}. We consider the following hypotheses:
\begin{itemize}
\item[(H$_1$)] There exists a constant $p\in(0, 1)$
such that the function $g(\cdot,\cdot)$ maps
$[0, a]\times H_{q}$ into $H_{p+q}$ and $A^{p}g: [0, a]\times H_{q}\rightarrow H_{q}$
satisfies a Lipschitz condition, that is, there exists a constant $L_{1}>0$ such that
\begin{equation}
\label{eq:3.1}
\Vert A^{p}g(t_{1}, u_{1})-A^{p}g(t_{2}, u_{2})\Vert_{q}
\leq L_{1}\left(\vert t_{1}-t_{2}\vert + \Vert u_{1}-u_{2}\Vert_{q}\right)
\end{equation}
for any $0\leq t_{1},t_{2}\leq a$, $u_{1}, u_{2}\in H_{q}$.
Moreover, there exists a constant $L_{2}>0$ such that the inequality
\begin{equation}
\label{eq:3.2}
\Vert A^{p}g(t, u)\Vert_{q}\leq L_{2}
\end{equation}
holds for any $u\in H_{q}$.
\item[(H$_2$)] The multi-valued map
$f: \Delta\times H_{q}\rightarrow P_{c,cp}(H)$ satisfies the following conditions:
\begin{itemize}
\item[(i)] function $f(t, s, \cdot): H_{q}\rightarrow P_{c,cp}(H)$ is u.s.c.
for each $(t, s)\in \Delta$, function $f(\cdot,\cdot, B_{1}\mu_{1})$ is measurable
for each $\mu_{1}\in L^{2}(J^{\prime}$, $U_{q})$, and the set
$$
S_{f, \mu_{1}}=\left\lbrace v\in L^{1}(J^{\prime}, H): B_{1}\mu_{1}(t)+v(\delta(t))
\in \int_{0}^{t}f(t, s, B_{1}\mu_{1}(\delta(s)))ds \text{ a.e. on } H_{q}\right\rbrace
$$
is nonempty;
\item[(ii)] there exists a positive constant $\omega$ such that
$\left\Vert \int_{0}^{t}f(t, s, B_{1}\mu_{1}(\delta(s)))ds\right\Vert \leq \omega$
for all $(t, s, \cdot)\in \Delta\times H_{q}$, where
$$
\left\Vert \int_{0}^{t}f(t, s, B_{1}\mu_{1}(\delta(s)))ds\right\Vert
=\sup\left\lbrace \Vert v\Vert:
B_{1}\mu_{1}+v\in \int_{0}^{t}f(t, s, B_{1}\mu_{1}(\delta(s)))ds\right\rbrace.
$$
\end{itemize}
\item[(H$_{3}$)] Function $h: C(J, H)\rightarrow H_{q}$
is a completely continuous map and there exists a positive constant
$k$ such that $\Vert h(u)\Vert_{q}\leq k$.
\item[(H$_{4}$)] The delay arguments $\sigma, \delta: J\rightarrow J^{\prime}$
are absolutely continuous and satisfy $\vert\sigma(t)\vert\leq t$,
$\vert\delta(t)\vert\leq t$, for every $t\in J$.
\end{itemize}
For $\lambda>0, u\in \Omega$, we define the operator
$F^{\lambda}: \Omega\rightarrow 2^{\Omega}$ as follows:
\begin{multline*}
F^{\lambda}(u)=\Biggl\lbrace z\in \Omega: z(t)
=S_{\alpha}(t)[B_{2}\mu_{2}(t)+u_{0}-h(u(t))
-g(0, u(\sigma(0)))]+g(t, u(\sigma(t)))\\
+\int_{0}^{t}(t-s)^{\alpha-1}\lbrace AT_{\alpha}(t-s)g(s, u(\sigma(s)))
+T_{\alpha}(t-s)[v(\delta(s))+B_{1}\mu_{1}(s)]\rbrace ds, v\in S_{f, \mu_{1}}\Biggr\rbrace.
\end{multline*}
For any $u(\cdot)\in \Omega, u_{a}\in H$, we take the controls
\begin{equation}
\label{eq:3.3}
\mu_{1}=B_{1}^{\ast}T_{\alpha}^{\ast}(a-t)
\mathcal{R}(\lambda, \Gamma^{a}_{0,1})P(u(\cdot)),
\quad \mu_{2}=B_{2}^{\ast}S_{\alpha}^{\ast}(a)
\mathcal{R}(\lambda, \Gamma^{a}_{0,2})P(u(\cdot)),
\end{equation}
where
\begin{multline}
\label{eq:3.4}
P\left(u(\cdot)\right)=u_{a}-S_{\alpha}(a)\left[u_{0}-h(u(t))
-g\left(0, u(\sigma(0))\right)\right]-g\left(t, u(\sigma(t))\right)\\
-\int_{0}^{a}(a-s)^{\alpha-1}\left\lbrace AT_{\alpha}(a-s)
g\left(s, u(\sigma(s))\right) + T_{\alpha}(a-s)v(\delta(s))\right\rbrace ds.
\end{multline}
We note that the fixed points of $F^{\lambda}$ are mild solutions
of the fractional nonlocal control inclusion \eqref{eq:1.1}--\eqref{eq:1.2}.

\begin{theorem}
\label{Theorem:3.1}
Let $u_{0}\in H_{q}$. If hypotheses (H$_1$)--(H$_4$) are satisfied, then
$F^{\lambda}$ has a fixed point on $J$ for each $\lambda>0$, provided
$\displaystyle L_{1}\left[M\Vert A^{-p}\Vert+\Vert A^{-p}\Vert+\frac{a^{p\alpha}}{p\alpha}
\frac{\alpha M_{1-p}\Gamma(1+p)}{\Gamma(1+\alpha p)}\right]<1$.
\end{theorem}

\begin{proof}
In order to prove the existence of mild solutions for system
\eqref{eq:1.1}--\eqref{eq:1.2}, we divide the proof into several steps.

\emph{Step 1.} For each $0<\lambda<1$, $F^{\lambda}(u)$ is bounded.
Using \eqref{eq:3.3}, Lemma~\ref{Lemma:2.5}, \eqref{eq:2.5} and \eqref{eq:2.6}, we get
\begin{equation*}
\Vert\mu_{1}\Vert\leq \frac{1}{\lambda}\frac{M}{\Gamma (\alpha)}\Vert B_{1}\Vert\Vert P(u(\cdot))\Vert,
\quad \Vert\mu_{2}\Vert\leq \frac{1}{\lambda}M\Vert B_{2}\Vert\Vert P(u(\cdot))\Vert.
\end{equation*}
Using \eqref{eq:3.4}, \eqref{eq:3.2},
Lemmas~\ref{Lemma:2.4}--\ref{Lemma:2.5}, and (H$_2$)--(H$_4$), we obtain that
\begin{equation*}
\begin{split}
\Vert P(u(\cdot))\Vert&\leq \Vert u_{a}\Vert
+\Vert A^{-q}\Vert\Biggl\lbrace \Vert S_{\alpha}(a)[u_{0}-h(u(t))-A^{-p}A^{p}g(0, u(\sigma(0)))]\Vert_{q}\\
&\quad +\Vert A^{-p}A^{p}g(t, u(\sigma(t)))\Vert_{q}\\
&\quad +\left\Vert\int_{0}^{a}(a-s)^{\alpha-1}A^{1-p}T_{\alpha}(a-s)A^{p}g(s, u(\sigma(s)))ds\right\Vert_{q}\\
&\quad +\left\Vert\int_{0}^{a}(a-s)^{\alpha-1}T_{\alpha}(a-s)v(\delta(s))ds\right\Vert_{q}\Biggr\rbrace\\
&\leq \Vert u_{a}\Vert+\Vert A^{-q}\Vert\Biggl\lbrace M[\Vert u_{0}\Vert_{q}+k+\Vert A^{-p}\Vert L_{2}]+\Vert A^{-p}\Vert L_{2}\\
&\quad + \int_{0}^{a}(a-s)^{\alpha-1}\frac{\alpha M_{1-p}\Gamma(1+p)}{\Gamma(1+\alpha p)}(a-s)^{-(1-p)\alpha}L_{2}ds\\
&\quad + \int_{0}^{a}(a-s)^{\alpha-1}\frac{\alpha M_{q}\Gamma(2-q)}{\Gamma(1+\alpha (1-q))}(a-s)^{-q\alpha}\omega ds\Biggr\rbrace\\
&\leq \Vert u_{a}\Vert+\Vert A^{-q}\Vert\biggl\lbrace M[\Vert u_{0}\Vert_{q}
+k+\Vert A^{-p}\Vert L_{2}]+\Vert A^{-p}\Vert L_{2}\\
&\quad + \frac{a^{p\alpha}}{p\alpha}\frac{\alpha M_{1-p}\Gamma(1+p)}{\Gamma(1+\alpha p)}L_{2}
+\frac{a^{\alpha(1-q)}}{\alpha(1-q)}\frac{\alpha M_{q}\Gamma(2-q)}{\Gamma(1+\alpha (1-q))}\omega\biggr\rbrace.
\end{split}
\end{equation*}
Now, for $z\in F^{\lambda}(u)$, we have
\begin{equation*}
\begin{split}
\Vert z(t)\Vert_{q}
&\leq \left\Vert S_{\alpha}(t)\left[B_{2}\mu_{2}(t)
+u_{0}-h(u(t))-A^{-p}A^{p}g(0, u(\sigma(0)))\right]\right\Vert_{q}\\
&\quad +\Vert A^{-p}A^{p}g(t, u(\sigma(t)))\Vert_{q}+\left\Vert
\int_{0}^{t}(t-s)^{\alpha-1} A^{1-p}T_{\alpha}(t-s)A^{p}g(s, u(\sigma(s)))ds\right\Vert_{q}\\
&\quad +\left\Vert\int_{0}^{t}(t-s)^{\alpha-1}T_{\alpha}(t-s)[v(\delta(s))+B_{1}\mu_{1}(s)] ds\right\Vert_{q}\\
&\leq M\left[\Vert B_{2}\Vert_{q}\Vert\mu_{2}\Vert+\Vert u_{0}\Vert_{q}+k+M_{p}L_{2}\right]+M_{p}L_{2}\\
&\quad + \int_{0}^{t}(t-s)^{\alpha-1}\frac{\alpha M_{1-p}\Gamma(1+p)}{\Gamma(1+\alpha p)}(t-s)^{-(1-p)\alpha}L_{2}ds\\
&\quad +\int_{0}^{t}(t-s)^{\alpha-1}\frac{\alpha M_{q}\Gamma(2-q)}{\Gamma(1+\alpha (1-q))}(t-s)^{-q\alpha}\left[
\omega+\Vert B_{1}\Vert\Vert\mu_{1}\Vert\right]ds\\
&\leq M\left[\Vert B_{2}\Vert_{q}\Vert\mu_{2}\Vert+\Vert u_{0}\Vert_{q}+k+\Vert A^{-p}\Vert L_{2}\right]
+\Vert A^{-p}\Vert L_{2}\\
&\quad + \frac{a^{p\alpha}}{p\alpha}\frac{\alpha M_{1-p}\Gamma(1+p)}{\Gamma(1+\alpha p)}L_{2}
+\frac{a^{\alpha(1-q)}}{\alpha(1-q)}\frac{\alpha M_{q}\Gamma(2-q)}{\Gamma(1
+\alpha (1-q))}\left[\omega+\Vert B_{1}\Vert\Vert\mu_{1}\Vert\right].
\end{split}
\end{equation*}
Thus, for every $u\in \Omega$, there exists a positive constant $r$ satisfying
$\Vert u\Vert_{\Omega}\leq r$. Hence, $F^{\lambda}(\Omega_{r})\subset\Omega_{r}$,
where $\Omega_{r}=\lbrace u\in \Omega: \Vert u\Vert_{\Omega}\leq r\rbrace$.

\emph{Step 2.} $F^{\lambda}(u)$ is convex for each $u\in \Omega_{r}$.
If $z_{1}, z_{2}\in F^{\lambda}(u)$, then there exists $v_{1}, v_{2}\in S_{f,\mu_{1}}$
such that, for each $t\in J$, we have
\begin{multline*}
z_{i}(t)=S_{\alpha}(t)\left[B_{2}\mu_{2,i}(t)+u_{0}-h(u(t))
-g(0, u(\sigma(0)))\right]+g\left(t, u(\sigma(t))\right)\\
+\int_{0}^{t}(t-s)^{\alpha-1}\lbrace AT_{\alpha}(t-s)g(s, u(\sigma(s)))
+T_{\alpha}(t-s)[v_{i}(\delta(s))+B_{1}\mu_{1,i}(s)]\rbrace ds,
\end{multline*}
where
\begin{multline*}
\mu_{1,i}=B_{1}^{\ast}T_{\alpha}^{\ast}(a-t)\mathcal{R}(\lambda, \Gamma^{a}_{0,1})
\biggl[u_{a}-S_{\alpha}(a)[u_{0}-h(u(t))-g(0, u(\sigma(0)))]-g(t, u(\sigma(t)))\\
-\int_{0}^{a}(a-s)^{\alpha-1}\left\lbrace AT_{\alpha}(a-s)g(s, u(\sigma(s)))
+T_{\alpha}(a-s)v_{i}(\delta(s))\right\rbrace ds\biggr],
\end{multline*}
\begin{multline*}
\mu_{2,i}=B_{2}^{\ast}S_{\alpha}^{\ast}(a)\mathcal{R}(\lambda, \Gamma^{a}_{0,2})\biggl[u_{a}
-S_{\alpha}(a)[u_{0}-h(u(t))-g(0, u(\sigma(0)))]-g(t, u(\sigma(t)))\\
-\int_{0}^{a}(a-s)^{\alpha-1}\left\lbrace AT_{\alpha}(a-s)g\left(s, u(\sigma(s))\right)
+T_{\alpha}(a-s)v_{i}(\delta(s))\right\rbrace ds\biggr],
\end{multline*}
$i=1,2$. Let $0\leq\beta\leq1$. Then,
\begin{equation*}
\begin{split}
\beta z_{1}(t)+(1-\beta)z_{2}(t)
=S_{\alpha}&(t)\left\lbrace B_{2}\left[\beta\mu_{2,1}(t)+(1-\beta)\mu_{2,2}(t)\right]
+u_{0}-h(u(t))-g(0, u(\sigma(0)))\right\rbrace\\
&+g\left(t, u(\sigma(t))\right)+\int_{0}^{t}(t-s)^{\alpha-1}AT_{\alpha}(t-s)g\left(s, u(\sigma(s))\right)ds\\
&+\int_{0}^{t}(t-s)^{\alpha-1}T_{\alpha}(t-s)\left[\beta v_{1}(\delta(s))+(1-\beta)v_{2}(\delta(s))\right]ds\\
&+\int_{0}^{t}(t-s)^{\alpha-1}T_{\alpha}(t-s)B_{1}\left[\beta\mu_{1,1}(s)+(1-\beta)\mu_{1,2}(s)\right] ds,
\end{split}
\end{equation*}
where
\begin{equation*}
\begin{split}
\beta\mu_{1,1}&(t)+(1-\beta)\mu_{1,2}(t)\\
&=B_{1}^{\ast}T_{\alpha}^{\ast}(a-t)\mathcal{R}(\lambda, \Gamma^{a}_{0,1})
\biggl[u_{a}-S_{\alpha}(a)[u_{0}-h(u(t))-g(0, u(\sigma(0)))]-g\left(t, u(\sigma(t))\right)\\
&-\int_{0}^{a}(a-s)^{\alpha-1}\left\lbrace AT_{\alpha}(a-s)g\left(s, u(\sigma(s))\right)
+T_{\alpha}(a-s)\left[\beta v_{1}(\delta(s))+(1-\beta)v_{2}(\delta(s))\right]\right\rbrace ds\biggr],
\end{split}
\end{equation*}
\begin{equation*}
\begin{split}
\beta\mu_{2,1}&(t)+(1-\beta)\mu_{2,2}(t)\\
&=B_{2}^{\ast}S_{\alpha}^{\ast}(a)\mathcal{R}(\lambda, \Gamma^{a}_{0,2})\biggl[
u_{a}-S_{\alpha}(a)[u_{0}-h(u(t))-g(0, u(\sigma(0)))]-g(t, u(\sigma(t)))\\
&-\int_{0}^{a}(a-s)^{\alpha-1}\left\lbrace AT_{\alpha}(a-s)g(s, u(\sigma(s)))
+T_{\alpha}(a-s)[\beta v_{1}(\delta(s))+(1-\beta)v_{2}(\delta(s))]\right\rbrace ds\biggr].
\end{split}
\end{equation*}
Since the multi-valued map $f$ has convex values, then $S_{f,\mu_{1}}$ is convex.
As required, we conclude that $\beta z_{1}(t)+(1-\beta)z_{2}(t)\in F^{\lambda}(u)$.

\emph{Step 3.} $F^{\lambda}(u)$ is closed for each $u\in \Omega_{r}$.
Let $\lbrace z_{n}\rbrace_{n\geq0}\in F^{\lambda}(u)$ for $z_{n}\rightarrow z\in \Omega_{r}$.
Then, there exists $v_{n}\in S_{f,\mu_{1}}$ such that
\begin{multline*}
z_{n}(t)=S_{\alpha}(t)\left[B_{2}\mu_{2,n}(t)+u_{0}-h(u(t))-g\left(0, u(\sigma(0))\right)\right]
+g\left(t, u(\sigma(t))\right)\\
+\int_{0}^{t}(t-s)^{\alpha-1}\lbrace AT_{\alpha}(t-s)g\left(s, u(\sigma(s))\right)
+T_{\alpha}(t-s)\left[v_{n}(\delta(s))+B_{1}\mu_{1,n}(s)\right]\rbrace ds
\end{multline*}
for every $t\in J$, where
\begin{multline*}
\mu_{1,n}=B_{1}^{\ast}T_{\alpha}^{\ast}(a-t)\mathcal{R}\left(\lambda, \Gamma^{a}_{0,1}\right)\biggl[
u_{a}-S_{\alpha}(a)[u_{0}-h(u(t))-g(0, u(\sigma(0)))]-g\left(t, u(\sigma(t))\right)\\
-\int_{0}^{a}(a-s)^{\alpha-1}\left\lbrace AT_{\alpha}(a-s)g(s, u(\sigma(s)))
+T_{\alpha}(a-s)v_{n}(\delta(s))\right\rbrace ds\biggr],
\end{multline*}
\begin{multline*}
\mu_{2,n}=B_{2}^{\ast}S_{\alpha}^{\ast}(a)\mathcal{R}(\lambda, \Gamma^{a}_{0,2})\biggl[
u_{a}-S_{\alpha}(a)[u_{0}-h(u(t))-g(0, u(\sigma(0)))]-g\left(t, u(\sigma(t))\right)\\
-\int_{0}^{a}(a-s)^{\alpha-1}\left\lbrace AT_{\alpha}(a-s)g\left(s, u(\sigma(s))\right)
+T_{\alpha}(a-s)v_{n}(\delta(s))\right\rbrace ds\biggr].
\end{multline*}
From \cite{AMA.35}, we deduce that $S_{f,\mu_{1}}$ is weakly compact in $L^{1}(J, H)$,
which implies that $v_{n}$ converges weakly ($\stackrel{\star}{\longrightarrow}$ for short)
to some $v\in S_{f,\mu_{1}}$ in $L^{1}(J, H)$. Therefore,
\begin{multline*}
\mu_{1,n}\stackrel{\star}{\longrightarrow}\mu_{1}
=B_{1}^{\ast}T_{\alpha}^{\ast}(a-t)\mathcal{R}\left(\lambda, \Gamma^{a}_{0,1}\right)\biggl[
u_{a}-S_{\alpha}(a)\left[u_{0}-h(u(t))-g(0, u(\sigma(0)))\right]-g\left(t, u(\sigma(t))\right)\\
-\int_{0}^{a}(a-s)^{\alpha-1}\left\lbrace AT_{\alpha}(a-s)g\left(s, u(\sigma(s))\right)
+T_{\alpha}(a-s)v(\delta(s))\right\rbrace ds\biggr],
\end{multline*}
\begin{multline*}
\mu_{2,n}\stackrel{\star}{\longrightarrow}\mu_{2}
=B_{2}^{\ast}S_{\alpha}^{\ast}(a)\mathcal{R}(\lambda, \Gamma^{a}_{0,2})\biggl[
u_{a}-S_{\alpha}(a)[u_{0}-h(u(t))-g(0, u(\sigma(0)))]-g\left(t, u(\sigma(t))\right)\\
-\int_{0}^{a}(a-s)^{\alpha-1}\left\lbrace AT_{\alpha}(a-s)g\left(s, u(\sigma(s))\right)
+T_{\alpha}(a-s)v(\delta(s))\right\rbrace ds\biggr].
\end{multline*}
Thus,
\begin{multline*}
z_{n}(t)\rightarrow z(t)=S_{\alpha}(t)\left[B_{2}\mu_{2}(t)
+u_{0}-h(u(t))-g\left(0, u(\sigma(0))\right)\right]+g\left(t, u(\sigma(t))\right)\\
+\int_{0}^{t}(t-s)^{\alpha-1}\lbrace AT_{\alpha}(t-s)g\left(s, u(\sigma(s))\right)
+T_{\alpha}(t-s)\left[v(\delta(s))+B_{1}\mu_{1}(s)\right]\rbrace ds.
\end{multline*}
Hence, $z\in F^{\lambda}(u)$.

\emph{Step 4.} $F^{\lambda}(u)$ is u.s.c. and condensing.
We make the decomposition $F^{\lambda}=F^{\lambda}_{1}+F^{\lambda}_{2}$,
where the operators $F^{\lambda}_{1}$ and $F^{\lambda}_{2}$ are defined by
$$
\left(F^{\lambda}_{1}u\right)(t)=-S_{\alpha}(t)g\left(0, u(\sigma(0))\right)
+g\left(t, u(\sigma(t))\right)+\int_{0}^{t}(t-s)^{\alpha-1}AT_{\alpha}(t-s)
g\left(s, u(\sigma(s))\right)ds
$$
and
\begin{multline*}
F^{\lambda}_{2}(u)=\biggl\lbrace z\in \Omega_{r}: z(t)
=S_{\alpha}(t)\left[B_{2}\mu_{2}(t)+u_{0}-h(u(t))\right]\\
+\int_{0}^{t}(t-s)^{\alpha-1}T_{\alpha}(t-s)\left[v(\delta(s))+B_{1}\mu_{1}(s)\right]ds,
\quad v\in S_{f, \mu_{1}}\biggr\rbrace.
\end{multline*}
We show that $F^{\lambda}_{1}$ is a contraction operator while
$F^{\lambda}_{2}$ is completely continuous. Let $u_{1}, u_{2}\in\Omega_{r}$.
Then, for each $t\in J$, condition \eqref{eq:3.1} gives
\begin{equation*}
\begin{split}
\Vert F^{\lambda}_{1}u_{1}(t)-F^{\lambda}_{1}u_{2}(t)\Vert_{q}
&\leq \Vert S_{\alpha}(t)A^{-p}\left[A^{p}g(0, u_{1}(\sigma(0)))
-A^{p}g(0, u_{2}(\sigma(0)))\right]\Vert_{q}\\
&\ +\Vert A^{-p}\left[A^{p}g(t, u_{1}(\sigma(t)))
-A^{p}g(t, u_{2}(\sigma(t)))\right]\Vert_{q}\\
&\ +\left\Vert\int_{0}^{t}(t-s)^{\alpha-1}A^{1-p}T_{\alpha}(t-s)A^{p}
\left[g(s, u_{1}(\sigma(s)))-g(s, u_{2}(\sigma(s)))\right]ds\right\Vert_{q}\\
&\leq \left[M\Vert A^{-p}\Vert L_{1}+\Vert A^{-p}\Vert L_{1}\right]
\sup\limits_{s\in J}\Vert u_{1}(s)-u_{2}(s)\Vert_{q}\\
&\ +\int_{0}^{t}(t-s)^{\alpha-1}\frac{\alpha M_{1-p}\Gamma(1+p)}{\Gamma(1+\alpha p)}(t-s)^{-(1-p)\alpha}ds
L_{1}\sup\limits_{s\in J}\Vert u_{1}(s)-u_{2}(s)\Vert_{q}\\
&\leq L_{1}\left[M\Vert A^{-p}\Vert+\Vert A^{-p}\Vert
+\frac{a^{p\alpha}}{p\alpha}\frac{\alpha M_{1-p}\Gamma(1+p)}{\Gamma(1+\alpha p)}\right]
\sup\limits_{s\in J}\Vert u_{1}(s)-u_{2}(s)\Vert_{q}.
\end{split}
\end{equation*}
Therefore, $\Vert F^{\lambda}_{1}u_{1}(t)-F^{\lambda}_{1}u_{2}(t)\Vert_{q}
\leq K\sup\limits_{s\in J}\Vert u_{1}(s)-u_{2}(s)\Vert_{q}$,
where $0\leq K<1$. Hence $F^{\lambda}_{1}$ is a contraction operator.
Next, we show that $F^{\lambda}_{2}$ is u.s.c. and completely continuous.
We begin to prove that $F^{\lambda}_{2}$ is completely continuous.

1. $F^{\lambda}_{2}$ is already bounded.

2. $F^{\lambda}_{2}$ is equicontinuous on $\Omega_{r}$.
Let $u\in \Omega_{r}$, $z\in (F^{\lambda}_{2})(u)$.
Then there exists $v\in S_{f,\mu_{1}}$ such that
$$
z(t)=S_{\alpha}(t)\left[B_{2}\mu_{2}(t)+u_{0}-h(u(t))\right]
+\int_{0}^{t}(t-s)^{\alpha-1}T_{\alpha}(t-s)\left[v(\delta(s))+B_{1}\mu_{1}(s)\right]ds
$$
for each $t\in J$. It follows that
\begin{equation*}
\begin{split}
\Vert z(\tau)-z(0)\Vert_{q}
&=\biggl\Vert (S_{\alpha}(\tau)-I)[A^{q}B_{2}\mu_{2}(t)+A^{q}u_{0}-A^{q}h(u(t))]\\
&\quad +\int_{0}^{\tau}(\tau-s)^{\alpha-1}A^{q}T_{\alpha}(\tau-s)[v(\delta(s))+B_{1}\mu_{1}(s)]ds\biggr\Vert\\
&\leq\left\Vert \left(S_{\alpha}(\tau)-I\right)\left[A^{q}B_{2}\mu_{2}(t)+A^{q}u_{0}-A^{q}h(u(t))\right]\right\Vert\\
&\quad + \frac{\tau^{\alpha(1-q)}}{\alpha(1-q)}\frac{\alpha M_{q}\Gamma(2-q)}{\Gamma(1+\alpha (1-q))}
\left[\omega+\Vert B_{1}\Vert\Vert\mu_{1}\Vert\right]\longrightarrow 0
\end{split}
\end{equation*}
as $\tau\rightarrow 0$ uniformly, since, by (H$_3$) and Lemma~\ref{Lemma:2.5},
the complete continuity of $A^{q}h$  and the strong continuity of $S_{\alpha}(t)$
at $t=0$ are satisfied, respectively. Let $\tau_{1}, \tau_{2}\in J$
with $0<s<\tau_{1}<\tau_{2}\leq a$. Then
\begin{equation*}
\begin{split}
\Vert z(\tau_{2})-z(\tau_{1})\Vert_{q}
&\leq\Vert \left[S_{\alpha}(\tau_{2})-S_{\alpha}(\tau_{1})\right]
\left[B_{2}\mu_{2}(t)+u_{0}-h(u(t))\right]\Vert_{q}\\
&\ +\left\Vert\int_{\tau_{1}}^{\tau_{2}}(\tau_{2}-s)^{\alpha-1}T_{\alpha}(\tau_{2}-s)
\left[v(\delta(s))+B_{1}\mu_{1}(s)\right]ds\right\Vert_{q}\\
&\ +\left\Vert\int_{0}^{\tau_{1}}\left[(\tau_{2}-s)^{\alpha-1}-(\tau_{1}-s)^{\alpha-1}\right]
T_{\alpha}(\tau_{2}-s)\left[v(\delta(s))+B_{1}\mu_{1}(s)\right]ds\right\Vert_{q}\\
&\ +\left\Vert\int_{0}^{\tau_{1}}(\tau_{1}-s)^{\alpha-1}
\left[T_{\alpha}(\tau_{2}-s)-T_{\alpha}(\tau_{1}-s)\right]
\left[v(\delta(s))+B_{1}\mu_{1}(s)\right]ds\right\Vert_{q}\\
&\leq\left\Vert S_{\alpha}(\tau_{2})-S_{\alpha}(\tau_{1})\right\Vert_{q}
\left[\left\Vert B_{2}\right\Vert \left\Vert\mu_{2}(t)\right\Vert
+\Vert u_{0}\Vert_{q}+\Vert h(u(t))\Vert\right]\\
&\ +\int_{\tau_{1}}^{\tau_{2}}(\tau_{2}-s)^{\alpha-1}\Vert A^{q}
T_{\alpha}(\tau_{2}-s)\Vert\left[\Vert v(\delta(s))\Vert+\Vert B_{1}\Vert\Vert\mu_{1}(s)\Vert\right]ds\\
&\ +\int_{0}^{\tau_{1}}[(\tau_{2}-s)^{\alpha-1}-(\tau_{1}-s)^{\alpha-1}]\Vert
A^{q}T_{\alpha}(\tau_{2}-s)\Vert\left[\Vert v(\delta(s))\Vert+\Vert B_{1}\Vert\Vert\mu_{1}(s)\Vert\right]ds\\
&\ +\int_{0}^{\tau_{1}}(\tau_{1}-s)^{\alpha-1}\Vert A^{q}
\left[T_{\alpha}(\tau_{2}-s)-T_{\alpha}(\tau_{1}-s)\Vert\right]
\left[\Vert v(\delta(s))\Vert+\Vert B_{1}\Vert\Vert\mu_{1}(s)\Vert\right]ds\\
&\leq\Vert S_{\alpha}(\tau_{2})-S_{\alpha}(\tau_{1})\Vert_{q}
\left[\Vert B_{2}\Vert\Vert\mu_{2}(t)\Vert+\Vert u_{0}\Vert_{q}+\Vert h(u(t))\Vert\right]\\
&\ +\int_{\tau_{1}}^{\tau_{2}}(\tau_{2}-s)^{\alpha-1}\Vert
A^{q}T_{\alpha}(\tau_{2}-s)\Vert\left[\omega+\Vert B_{1}\Vert\Vert\mu_{1}(s)\Vert\right]ds\\
&\ +\int_{0}^{\tau_{1}-\epsilon}\left[(\tau_{2}-s)^{\alpha-1}-(\tau_{1}-s)^{\alpha-1}\right]\Vert
A^{q}T_{\alpha}(\tau_{2}-s)\Vert\left[\omega+\Vert B_{1}\Vert\Vert\mu_{1}(s)\Vert\right]ds\\
&\ +\int_{\tau_{1}-\epsilon}^{\tau_{1}}\left[(\tau_{2}-s)^{\alpha-1}-(\tau_{1}-s)^{\alpha-1}\right]\Vert
A^{q}T_{\alpha}(\tau_{2}-s)\Vert\left[\omega+\Vert B_{1}\Vert\Vert\mu_{1}(s)\Vert\right]ds\\
&\ +\int_{0}^{\tau_{1}-\epsilon}(\tau_{1}-s)^{\alpha-1}\Vert A^{q}
\left[T_{\alpha}(\tau_{2}-s)-T_{\alpha}(\tau_{1}-s)\Vert\right]
\left[\omega+\Vert B_{1}\Vert\Vert\mu_{1}(s)\Vert\right]ds\\
&\ +\int_{\tau_{1}-\epsilon}^{\tau_{1}}(\tau_{1}-s)^{\alpha-1}\Vert
A^{q}\left[T_{\alpha}(\tau_{2}-s)-T_{\alpha}(\tau_{1}-s)\Vert\right]
\left[\omega+\Vert B_{1}\Vert\Vert\mu_{1}(s)\Vert\right]ds
\end{split}
\end{equation*}
\begin{equation*}
\begin{split}
&\leq\Vert S_{\alpha}(\tau_{2})-S_{\alpha}(\tau_{1})\Vert_{q}
\left[\Vert B_{2}\Vert\Vert\mu_{2}(t)\Vert+\Vert u_{0}\Vert_{q}+\Vert h(u(t))\Vert\right]\\
&\ +\frac{\alpha M_{q}\Gamma(2-q)}{\Gamma(1+\alpha (1-q))}\Biggl[
\frac{(\tau_{2}-\tau_{1})^{\alpha(1-q)}}{\alpha(1-q)}
\left[\omega+\Vert B_{1}\Vert\Vert\mu_{1}\Vert\right]\\
&\ +\int_{0}^{\tau_{1}-\epsilon}\left[(\tau_{2}-s)^{\alpha-1}
-(\tau_{1}-s)^{\alpha-1}\right](\tau_{2}-s)^{-q\alpha}
\left[\omega+\Vert B_{1}\Vert\Vert\mu_{1}(s)\Vert\right]ds\\
&\ +\int_{\tau_{1}-\epsilon}^{\tau_{1}}\left[(\tau_{2}-s)^{\alpha-1}
-(\tau_{1}-s)^{\alpha-1}\right](\tau_{2}-s)^{-q\alpha}
\left[\omega+\Vert B_{1}\Vert\Vert\mu_{1}(s)\Vert\right]ds\\
&\ +\int_{0}^{\tau_{1}-\epsilon}(\tau_{1}-s)^{\alpha-1}
\left[(\tau_{2}-s)^{-q\alpha}-(\tau_{1}-s)^{-q\alpha}\right]
\left[\omega+\Vert B_{1}\Vert\Vert\mu_{1}(s)\Vert\right]ds\\
&\ +\int_{\tau_{1}-\epsilon}^{\tau_{1}}(\tau_{1}-s)^{\alpha-1}
\left[(\tau_{2}-s)^{-q\alpha}-(\tau_{1}-s)^{-q\alpha}\right]
\left[\omega+\Vert B_{1}\Vert\Vert\mu_{1}(s)\Vert\right]ds\Biggr].
\end{split}
\end{equation*}
In view of Lemma~\ref{Lemma:2.5}, $S_{\alpha}(\cdot)$ and
$T_{\alpha}(\cdot)$ are compact and strongly continuous operators,
which imply the continuity of those operators in the uniform operator topology on $(0, a]$.
Concluding, as $\tau_{2}-\tau_{1}\rightarrow 0$, with $\epsilon$ sufficiently small,
the right-hand side of the above inequality tends to zero independently of $u\in\Omega_{r}$.
This shows the equicontinuity of $F^{\lambda}_{2}$ on $\Omega_{r}$.

3. $(F^{\lambda}_{2}\Omega_{r})(t)=\lbrace z(t): z\in F^{\lambda}_{2}(\Omega_{r})\rbrace$
is relatively compact in $H_{q}$ for each $t\in J$.
Clearly, $(F^{\lambda}_{2}\Omega_{r})(t)$ is relatively compact in $H_{q}$ for $t=0$.
Let $0<t\leq a$ be fixed. For $u\in \Omega_{r}$ and $z\in (F^{\lambda}_{2})(u)$, there exists
a function $v\in S_{f,\mu_{1}}$ such that
$$
z(t)=S_{\alpha}(t)\left[B_{2}\mu_{2}(t)+u_{0}-h(u(t))\right]
+\int_{0}^{t}(t-s)^{\alpha-1}T_{\alpha}(t-s)\left[v(\delta(s))+B_{1}\mu_{1}(s)\right]ds.
$$
For $0<\gamma<1$, we have
\begin{equation*}
\begin{split}
\Vert A^{\gamma}z(t)\Vert
&\leq \Vert S_{\alpha}(t)\Vert\Vert A^{\gamma}\left[B_{2}\mu_{2}(t)+u_{0}-h(u(t))\right]\Vert\\
&\quad +\int_{0}^{t}(t-s)^{\alpha-1}\Vert A^{\gamma}T_{\alpha}(t-s)\Vert\Vert v(\delta(s))+B_{1}\mu_{1}(s)\Vert ds\\
&\leq \Vert S_{\alpha}(t)\Vert \Vert A^{\gamma}A^{-q}\left[ B_{2}\mu_{2}(t)+u_{0}-h(u(t))\Vert_{q}\right]\\
&\quad +\int_{0}^{t}(t-s)^{\alpha-1}\Vert A^{\gamma}A^{-q}T_{\alpha}(t-s)\Vert_{q}
\left[\Vert v(\delta(s))\Vert+\Vert B_{1}\Vert\Vert\mu_{1}(s)\Vert\right]ds\\
&\leq M\Vert A^{-q\gamma}\Vert\left[\Vert B_{2}\Vert\Vert\mu_{2}\Vert_{q}+\Vert u_{0}\Vert_{q}+k\right]\\
&\quad +\Vert A^{-q\gamma}\Vert\frac{a^{\alpha(1-q)}}{\alpha(1-q)}
\frac{\alpha M_{q}\Gamma(2-q)}{\Gamma(1+\alpha (1-q))}\left[\omega+\Vert B_{1}\Vert\Vert\mu_{1}\Vert\right].
\end{split}
\end{equation*}
By Remark~\ref{Remark:2.2}, $A^{-q\gamma}$ is bounded since $0<q\gamma<1$.
Clearly, $A^{\gamma}F^{\lambda}_{2}u(t)$ is bounded in $H$.
It is known from \cite{AMA.36} that $A^{-\gamma}: H\rightarrow H_{q}$ is compact
for $0\leq q<\gamma<1$. Then $(F^{\lambda}_{2}\Omega_{r})(t)$
is relatively compact in $H_{q}$ for each $t\in J$.
As a consequence of Step~4.1, together with the Arzela--Ascoli theorem,
we conclude that $F^{\lambda}_{2}$ is completely continuous.

4. $F^{\lambda}_{2}$ has a closed graph. From above we have that $F^{\lambda}_{2}(u)$
is a relatively compact and closed set for every $u\in \Omega_{r}$.
Hence $F^{\lambda}_{2}(u)$ is a compact set.
Let $u_{n}\rightarrow u_{*}$, $u_{n}\in \Omega_{r}$, $z_{n}\in F^{\lambda}_{2}(u_{n})$
and $z_{n}\rightarrow z_{*}$. We shall prove that $z_{*}\in F^{\lambda}_{2}(u_{*})$.
Note that $z_{n}\in F^{\lambda}_{2}(u_{n})$, which means that there exists
$v_{n}\in S_{f,\mu_{1,n}}$ such that
\begin{equation}
\label{eq:3.5}
z_{n}(t)=S_{\alpha}(t)\left[B_{2}\mu_{2,n}(t)+u_{0}-h(u_{n})\right]
+\int_{0}^{t}(t-s)^{\alpha-1}T_{\alpha}(t-s)\left[v_{n}(\delta(s))+B_{1}\mu_{1,n}(s)\right]ds
\end{equation}
for each $t\in J$, where
\begin{multline*}
\mu_{1,n}=B_{1}^{\ast}T_{\alpha}^{\ast}(a-t)\mathcal{R}\left(\lambda, \Gamma^{a}_{0,1}\right)
\biggl[u_{a}-S_{\alpha}(a)\left[u_{0}-h(u_{n})-g\left(0, u_{n}(\sigma(0))\right)\right]
-g\left(a, u_{n}(\sigma(a))\right)\\
-\int_{0}^{a}(a-s)^{\alpha-1}\lbrace AT_{\alpha}(a-s)g\left(s, u_{n}(\sigma(s))\right)
+T_{\alpha}(a-s)v_{n}(\delta(s))\rbrace ds\biggr],
\end{multline*}
\begin{multline*}
\mu_{2,n}=B_{2}^{\ast}S_{\alpha}^{\ast}(a)\mathcal{R}\left(\lambda, \Gamma^{a}_{0,2}\right)
\biggl[u_{a}-S_{\alpha}(a)\left[u_{0}
-h(u_{n})-g(0, u_{n}(\sigma(0)))\right]-g\left(a, u_{n}(\sigma(a))\right)\\
-\int_{0}^{a}(a-s)^{\alpha-1}\lbrace AT_{\alpha}(a-s)g\left(s, u_{n}(\sigma(s))\right)
+T_{\alpha}(a-s)v_{n}(\delta(s))\rbrace ds\biggr].
\end{multline*}
We prove the existence of $v_{*}\in S_{f,\mu_{1,*}}$ such that
$$
z_{*}(t)=S_{\alpha}(t)\left[B_{2}\mu_{2,*}(t)+u_{0}-h(u_{*})\right]
+\int_{0}^{t}(t-s)^{\alpha-1}T_{\alpha}(t-s)\left[v_{*}(\delta(s))+B_{1}\mu_{1,*}(s)\right]ds
$$
for each $t\in J$, where
\begin{multline*}
\mu_{1,*}=B_{1}^{\ast}T_{\alpha}^{\ast}(a-t)\mathcal{R}\left(\lambda, \Gamma^{a}_{0,1}\right)\biggl[
u_{a}-S_{\alpha}(a)\left[u_{0}-h(u_{*})
-g\left(0, u_{*}(\sigma(0))\right)\right]-g\left(a, u_{*}(\sigma(a))\right)\\
-\int_{0}^{a}(a-s)^{\alpha-1}\lbrace AT_{\alpha}(a-s)g\left(s, u_{*}(\sigma(s))\right)
+T_{\alpha}(a-s)v_{*}(\delta(s))\rbrace ds\biggr],
\end{multline*}
\begin{multline*}
\mu_{2,*}=B_{2}^{\ast}S_{\alpha}^{\ast}(a)\mathcal{R}(\lambda, \Gamma^{a}_{0,2})
\biggl[u_{a}-S_{\alpha}(a)\left[u_{0}-h(u_{*})-g\left(0, u_{*}(\sigma(0))\right)\right]
-g\left(a, u_{*}(\sigma(a))\right)\\
-\int_{0}^{a}(a-s)^{\alpha-1}\lbrace AT_{\alpha}(a-s)
g\left(s, u_{*}(\sigma(s))\right)+T_{\alpha}(a-s)v_{*}(\delta(s))\rbrace ds\biggr].
\end{multline*}
Consider the linear continuous operator $P: L^{1}(J^{\prime}, H)\rightarrow L^{2}(J^{\prime}, H)$
defined by
\begin{multline*}
v\rightarrow P(v)(t)=\int_{0}^{t}(t-s)^{\alpha-1}T_{\alpha}(t-s)\\
\times\left[v(\delta(s))-B_{1}B_{1}^{\ast}T_{\alpha}^{\ast}(a-s)\mathcal{R}(\lambda, \Gamma^{a}_{0,1})
\int_{0}^{a}(a-\eta)^{\alpha-1}T_{\alpha}(a-\eta)v(\delta(\eta))d\eta\right]ds.
\end{multline*}
By \eqref{eq:3.5} we get
\begin{multline*}
z_{n}(t)-S_{\alpha}(t)\left[B_{2}\mu_{2,n}(t)+u_{0}-h(u_{n})\right]
-\int_{0}^{t}(t-s)^{\alpha-1}T_{\alpha}(t-s)B_{1}B_{1}^{\ast}
T_{\alpha}^{\ast}(a-s)\mathcal{R}\left(\lambda, \Gamma^{a}_{0,1}\right)\\
\times\biggl\lbrace u_{a}-S_{\alpha}(a)\left[u_{0}-h(u_{n})-g(0, u_{n}(\sigma(0)))\right]
-g\left(a, u_{n}(\sigma(a))\right)\\
-\int_{0}^{a}(a-\eta)^{\alpha-1}AT_{\alpha}(a-\eta)g\left(\eta, u_{n}\left(\sigma(\eta)\right)\right)
d\eta\biggr\rbrace ds\in P\left(S_{f,\mu_{1,n}}\right),
\end{multline*}
which converges to
\begin{multline*}
z_{*}(t)-S_{\alpha}(t)\left[B_{2}\mu_{2,*}(t)+u_{0}-h(u_{*})\right]
-\int_{0}^{t}(t-s)^{\alpha-1}T_{\alpha}(t-s)B_{1}
B_{1}^{\ast}T_{\alpha}^{\ast}(a-s)\mathcal{R}\left(\lambda, \Gamma^{a}_{0,1}\right)\\
\times\biggl\lbrace u_{a}-S_{\alpha}(a)\left[u_{0}-h(u_{*})
-g\left(0, u_{*}(\sigma(0))\right)\right]-g\left(a, u_{*}(\sigma(a))\right)\\
-\int_{0}^{a}(a-\eta)^{\alpha-1}AT_{\alpha}(a-\eta)g\left(\eta, u_{*}(\sigma(\eta))\right)d\eta\biggr\rbrace ds
\end{multline*}
in $\Omega_{r}$, uniformly as $n\rightarrow \infty$.
From Proposition~\ref{Proposition:2.2}, it follows that $P\circ S_{f}$
is a closed graph operator. Hence we have that
\begin{multline*}
z_{*}(t)-S_{\alpha}(t)\left[B_{2}\mu_{2,*}(t)+u_{0}-h(u_{*})\right]
-\int_{0}^{t}(t-s)^{\alpha-1}T_{\alpha}(t-s)
B_{1}B_{1}^{\ast}T_{\alpha}^{\ast}(a-s)\mathcal{R}\left(\lambda, \Gamma^{a}_{0,1}\right)\\
\times\biggl\lbrace u_{a}-S_{\alpha}(a)\left[u_{0}-h(u_{*})-g\left(0, u_{*}(\sigma(0))\right)\right]
-g\left(a, u_{*}(\sigma(a))\right)\\
-\int_{0}^{a}(a-\eta)^{\alpha-1}AT_{\alpha}(a-\eta)g\left(\eta, u_{*}(\sigma(\eta))\right)
d\eta\biggr\rbrace ds\in P\left(S_{f,\mu_{1,*}}\right),
\end{multline*}
that is, there exists $v_{*}(t)\in S_{f,\mu_{1,*}}$ such that
\begin{equation*}
\begin{split}
P(v_{*}(t)) &= z_{*}(t)-S_{\alpha}(t)\left[B_{2}\mu_{2,*}(t)+u_{0}-h(u_{*})\right]\\
&\quad -\int_{0}^{t}(t-s)^{\alpha-1}T_{\alpha}(t-s)B_{1}B_{1}^{\ast}
T_{\alpha}^{\ast}(a-s)\mathcal{R}\left(\lambda, \Gamma^{a}_{0,1}\right)\\
&\quad \times\biggl\lbrace u_{a}-S_{\alpha}(a)\left[u_{0}-h(u_{*})-g(0, u_{*}(\sigma(0)))\right]
-g\left(a, u_{*}(\sigma(a))\right)\\
&\quad -\int_{0}^{a}(a-\eta)^{\alpha-1}AT_{\alpha}(a-\eta)
g\left(\eta, u_{*}(\sigma(\eta))\right)d\eta\biggr\rbrace ds\\
&=\int_{0}^{t}(t-s)^{\alpha-1}T_{\alpha}(t-s)\\
&\quad \times\left[v_{*}(\delta(s))-B_{1}B_{1}^{\ast}T_{\alpha}^{\ast}(a-s)\mathcal{R}\left(\lambda,
\Gamma^{a}_{0,1}\right) \int_{0}^{a}(a-\eta)^{\alpha-1}
T_{\alpha}(a-\eta)v_{*}\left(\delta(\eta)\right)d\eta\right]ds.
\end{split}
\end{equation*}
This shows that $z_{*}\in F^{\lambda}_{2}(u_{*})$. Therefore,
$F^{\lambda}_{2}$ has a closed graph and $F^{\lambda}_{2}$
is a completely continuous multi-valued map with compact value.
Thus, $F^{\lambda}_{2}$ is u.s.c. On the other hand, $F^{\lambda}_{1}$
is proved a contraction operator and hence $F^{\lambda}=F^{\lambda}_{1}+F^{\lambda}_{2}$
is u.s.c. and condensing. According to Lemma~\ref{Lemma:2.3}, we ensure the existence
of a fixed point $u^{\lambda}(\cdot)$ for $F^{\lambda}$ in $\Omega_{r}$.
\end{proof}

\begin{theorem}
\label{Theorem:3.2}
If (H$_1$)--(H$_4$) are satisfied and
$\lambda\mathcal{R}(\lambda, \Gamma^{a}_{0,i})\rightarrow 0$
in the strong operator topology as $\lambda\rightarrow 0^{+}$, $i=1,2$,
then the nonlocal-control fractional delay system \eqref{eq:1.1}--\eqref{eq:1.2}
is approximately controllable on $J$.
\end{theorem}

\begin{proof}
According to Theorem~\ref{Theorem:3.1}, $F^{\lambda}$ has a fixed point in $\Omega_{r}$
for any $\lambda\in (0, 1)$. This implies that there exists
$\overline{u}^{\lambda}\in F^{\lambda}(\overline{u}^{\lambda})$, that is, there is
$\overline{v}^{\lambda}\in S_{f,\overline{\mu_{1}}^{\lambda}}$ such that
\begin{multline*}
\overline{u}^{\lambda}(t)=S_{\alpha}(t)\left[
B_{2}\overline{\mu_{2}}^{\lambda}(t)+u_{0}-h\left(\overline{u}^{\lambda}\right)
-g\left(0, \overline{u}^{\lambda}(\sigma(0))\right)\right]
+g\left(t, \overline{u}^{\lambda}(\sigma(t))\right)\\
+\int_{0}^{t}(t-s)^{\alpha-1}\lbrace AT_{\alpha}(t-s)
g\left(s, \overline{u}^{\lambda}(\sigma(s))\right)
+T_{\alpha}(t-s)\left[\overline{v}^{\lambda}(\delta(s))
+B_{1}\overline{\mu_{1}}^{\lambda}(s)\right]\rbrace ds,
\end{multline*}
where
\begin{multline*}
\overline{\mu_{1}}^{\lambda}(t)=B_{1}^{\ast}
T_{\alpha}^{\ast}(a-t)\mathcal{R}\left(\lambda, \Gamma^{a}_{0,1}\right)
\Biggl[u_{a}-S_{\alpha}(a)\left[u_{0}-h(\overline{u}^{\lambda}(t))
-g\left(0, \overline{u}^{\lambda}(\sigma(0))\right)\right]
-g\left(t, \overline{u}^{\lambda}(\sigma(t))\right)\\
-\int_{0}^{a}(a-s)^{\alpha-1}\left\lbrace AT_{\alpha}(a-s)
g\left(s, \overline{u}^{\lambda}(\sigma(s))\right)
+T_{\alpha}(a-s)\overline{v}^{\lambda}(\delta(s))\right\rbrace ds\Biggr],
\end{multline*}
\begin{multline*}
\overline{\mu_{2}}^{\lambda}(t)
=B_{2}^{\ast}S_{\alpha}^{\ast}(a)\mathcal{R}\left(\lambda, \Gamma^{a}_{0,2}\right)
\Biggl[u_{a}-S_{\alpha}(a)\left[u_{0}-h(\overline{u}^{\lambda}(t))
-g\left(0, \overline{u}^{\lambda}(\sigma(0))\right)\right]
-g\left(t, \overline{u}^{\lambda}(\sigma(t))\right)\\
-\int_{0}^{a}(a-s)^{\alpha-1}\left\lbrace AT_{\alpha}(a-s)
g\left(s, \overline{u}^{\lambda}(\sigma(s))\right)
+T_{\alpha}(a-s)\overline{v}^{\lambda}(\delta(s))\right\rbrace ds\Biggr].
\end{multline*}
Now,
\begin{multline*}
\overline{u}^{\lambda}(a)
=S_{\alpha}(a)\left[B_{2}\overline{\mu_{2}}^{\lambda}(a)+u_{0}
-h\left(\overline{u}^{\lambda}\right)
-g\left(0, \overline{u}^{\lambda}(\sigma(0))\right)\right]
+g\left(a, \overline{u}^{\lambda}(\sigma(a))\right)\\
+\int_{0}^{a}(a-s)^{\alpha-1}\lbrace AT_{\alpha}(a-s)
g\left(s, \overline{u}^{\lambda}(\sigma(s))\right)
+T_{\alpha}(a-s)\left[\overline{v}^{\lambda}(\delta(s))
+B_{1}\overline{\mu_{1}}^{\lambda}(s)\right]\rbrace ds
\end{multline*}
and
\begin{equation*}
\begin{split}
u_{a}-\overline{u}^{\lambda}(a)
=u_{a}&-\Gamma^{a}_{0,2}\mathcal{R}\left(\lambda, \Gamma^{a}_{0,2}\right)\lbrace
u_{a}-S_{\alpha}(a)\left[u_{0}-h\left(\overline{u}^{\lambda}\right)
-g\left(0, \overline{u}^{\lambda}(\sigma(0))\right)\right]
-g\left(a, \overline{u}^{\lambda}(\sigma(a))\right)\\
&-\int_{0}^{a}(a-s)^{\alpha-1}\lbrace AT_{\alpha}(a-s)
g\left(s, \overline{u}^{\lambda}(\sigma(s))\right)
+T_{\alpha}(a-s)\overline{v}^{\lambda}(\delta(s))\rbrace ds\rbrace\\
&-S_{\alpha}(a)\left[u_{0}-h\left(\overline{u}^{\lambda}\right)
-g\left(0, \overline{u}^{\lambda}(\sigma(0))\right)\right]
-g\left(a, \overline{u}^{\lambda}(\sigma(a))\right)\\
&-\int_{0}^{a}(a-s)^{\alpha-1}\lbrace AT_{\alpha}(a-s)
g\left(s, \overline{u}^{\lambda}(\sigma(s))\right)
+T_{\alpha}(a-s)\overline{v}^{\lambda}(\delta(s))\rbrace ds,\\
&-\Gamma^{a}_{0,1}\mathcal{R}\left(\lambda, \Gamma^{a}_{0,1}\right)\lbrace u_{a}
-S_{\alpha}(a)\left[u_{0}-h\left(\overline{u}^{\lambda}\right)
-g\left(0, \overline{u}^{\lambda}(\sigma(0))\right)\right]
-g\left(a, \overline{u}^{\lambda}(\sigma(a))\right)\\
&-\int_{0}^{a}(a-s)^{\alpha-1}\lbrace AT_{\alpha}(a-s)
g\left(s, \overline{u}^{\lambda}(\sigma(s))\right)
+T_{\alpha}(a-s)\overline{v}^{\lambda}(\delta(s))\rbrace ds\rbrace.
\end{split}
\end{equation*}
From \eqref{eq:2.6} we have
$I-\Gamma^{a}_{0,i}\mathcal{R}\left(\lambda, \Gamma^{a}_{0,i}\right)
=\lambda\mathcal{R}\left(\lambda, \Gamma^{a}_{0,i}\right)$, $i=1,2$,
and we deduce that
\begin{multline}
\label{eq:3.6}
u_{a}-\overline{u}^{\lambda}(a)
=\lambda\left[\mathcal{R}\left(\lambda, \Gamma^{a}_{0,1}\right)
+\mathcal{R}\left(\lambda, \Gamma^{a}_{0,2}\right)\right]\Biggl\lbrace u_{a}
-S_{\alpha}(a)\left[u_{0}-h\left(\overline{u}^{\lambda}\right)
-g\left(0, \overline{u}^{\lambda}(\sigma(0))\right)\right]\\
-g\left(a, \overline{u}^{\lambda}(\sigma(a))\right)
-\int_{0}^{a}(a-s)^{\alpha-1}\left\lbrace AT_{\alpha}(a-s)
g\left(s, \overline{u}^{\lambda}(\sigma(s))\right)
+T_{\alpha}(a-s)\overline{v}^{\lambda}(\delta(s))\right\rbrace ds\Biggr\rbrace.
\end{multline}
According to the compactness of $S_{\alpha}(t)$, $t>0$,
and the uniform boundedness of $h$ and $g$, we see that there
is a subsequence of $S_{\alpha}(a)\left[h\left(\overline{u}^{\lambda}\right)
-g\left(0, \overline{u}^{\lambda}(\sigma(0))\right)\right]$
that converges to some $u_{1}$ as $\lambda\rightarrow 0^{+}$.
By assumption (H$_1$), we can choose a sufficiently small positive constant
$\epsilon>0$, $q+\epsilon<1$, such that
$A^{q+\epsilon}g(a, \overline{u}^{\lambda}(\sigma(a)))$ is bounded in $H$. Since
$g\left(a, \overline{u}^{\lambda}(\sigma(a))\right)
=A^{-(q+\epsilon)}A^{q+\epsilon}g\left(a, \overline{u}^{\lambda}(\sigma(a))\right)$,
$g\left(a, \overline{u}^{\lambda}(\sigma(a))\right)$ is clearly relatively compact
in $H_{q}$ and hence in $H$ ($A^{-(q+\epsilon)}: H\rightarrow H_{q}$ is compact).
It means that there is $u_{2}\in H$ such that
$g\left(a, \overline{u}^{\lambda}(\sigma(a))\right)\rightarrow u_{2}$
in $\Vert\cdot\Vert$ as $\lambda\rightarrow 0^{+}$ (here $g$ is a subsequence of itself).
On the other hand,
$$
\int_{0}^{a}(a-s)^{\alpha-1}AT_{\alpha}(a-s)
g\left(s, \overline{u}^{\lambda}(\sigma(s))\right)ds
=\int_{0}^{a}(a-s)^{\alpha-1}A^{1-q}
T_{\alpha}(a-s)A^{q} g\left(s, \overline{u}^{\lambda}(\sigma(s))\right)ds
$$
and, by (H$_1$),
$A^{q}g(s, \overline{u}^{\lambda}(\sigma(s)))\in L^{2}(J, H)$.
Then we can get a subsequence, still denoted by
$A^{q}g(s, \overline{u}^{\lambda}(\sigma(s)))$, which converges weakly
to some $g(s)\in L^{2}(J, H)$. Similarly as the proof of the compactness
of $F^{\lambda}_{2}$ in Theorem~\ref{Theorem:3.1}, it is easy to see that the mapping
$$
u(t)\rightarrow\int_{0}^{t}(t-s)^{\alpha-1}A^{1-q}T_{\alpha}(t-s)u(s)ds,
$$
from $L^{2}(J, H)$ to $C(J, H)$, is compact. Then,
$$
\int_{0}^{a}(a-s)^{\alpha-1}A^{1-q}T_{\alpha}(a-s)
\left[A^{q}g(s, \overline{u}^{\lambda}(\sigma(s)))-g(s)\right]ds \rightarrow 0
$$
as $\lambda\rightarrow 0^{+}$. Similarly,
$\overline{v}^{\lambda}(\delta(s))$ is uniformly bounded
in $L^{2}(J^{\prime}, H)$ and so converges weakly to some $v(s)\in L^{2}(J, H)$.
The mapping
$$
u(t)\rightarrow\int_{0}^{t}(t-s)^{\alpha-1}T_{\alpha}(t-s)u(s)ds
$$
is also compact on $L^{2}(J, H)$. It follows that
$$
\int_{0}^{a}(a-s)^{\alpha-1}T_{\alpha}(a-s)\left[\overline{v}^{\lambda}(s)-v(s)\right]ds
\rightarrow 0
$$
as $\lambda \rightarrow 0^{+}$. Using \eqref{eq:3.6}, we get
\begin{equation*}
\begin{split}
\left\Vert u_{a}-\overline{u}^{\lambda}(a)\right\Vert
&=\Biggl\Vert\lambda\left[\mathcal{R}\left(\lambda, \Gamma^{a}_{0,1}\right)
+\mathcal{R}\left(\lambda, \Gamma^{a}_{0,2}\right)\right]\Bigl\lbrace u_{a}
-S_{\alpha}(a)\left[u_{0}-h\left(\overline{u}^{\lambda}\right)
-g\left(0, \overline{u}^{\lambda}(\sigma(0))\right)\right]\\
&\quad -g\left(a, \overline{u}^{\lambda}(\sigma(a))\right)
+T_{\alpha}(a-s)\overline{v}^{\lambda}(\delta(s))\Bigr\rbrace\\
&\quad -\int_{0}^{a}(a-s)^{\alpha-1}\left\lbrace AT_{\alpha}(a-s)
g\left(s, \overline{u}^{\lambda}(\sigma(s))\right)\right\rbrace ds\Biggr\Vert\\
&\leq\Biggl\Vert\lambda\left[\mathcal{R}\left(\lambda, \Gamma^{a}_{0,1}\right)
+\mathcal{R}\left(\lambda, \Gamma^{a}_{0,2}\right)\right]\Bigl\lbrace u_{a}
-S_{\alpha}(a)\left[u_{0}-u_{1}\right]-u_{2}\\
&\quad -\int_{0}^{a}(a-s)^{\alpha-1}\lbrace A^{q} T_{\alpha}(a-s)g(s)
+T_{\alpha}(a-s)v(\delta(s))\rbrace ds\Bigr\rbrace\Biggr\Vert\\
&\quad +\left\Vert\lambda\left[\mathcal{R}\left(\lambda, \Gamma^{a}_{0,1}\right)
+\mathcal{R}\left(\lambda, \Gamma^{a}_{0,2}\right)\right]\right\Vert
\Biggl\lbrace \left\Vert S_{\alpha}(a)\left[h\left(\overline{u}^{\lambda}\right)
-g\left(0, \overline{u}^{\lambda}(\sigma(0))\right)-u_{1}\right]\right\Vert\\
&\quad +\Biggl\Vert\int_{0}^{a}(a-s)^{\alpha-1}\Bigl\lbrace A^{1-q}T_{\alpha}(a-s)
\left[A^{q}g(s, \overline{u}^{\lambda}(\sigma(s)))-g(s)\right]\\
&\quad +\Vert g\left(a, \overline{u}^{\lambda}(\sigma(a))\right)-u_{2}\Vert
+T_{\alpha}(a-s)\left[\overline{v}^{\lambda}(\delta(s))-v(s)\right]\Bigr\rbrace ds\Biggr\Vert\Biggr\rbrace.
\end{split}
\end{equation*}
Moreover, by the assumption that we have $\lambda\mathcal{R}(\lambda, \Gamma^{a}_{0,i})\rightarrow 0$
in the strong operator topology as $\lambda\rightarrow 0^{+}$, $i=1,2$,
we ensure that
$\Vert u_{a}-\overline{u}^{\lambda}(a)\Vert\rightarrow 0$
as $\lambda\rightarrow 0^{+}$. Therefore, the fractional dynamic inclusion
\eqref{eq:1.1}--\eqref{eq:1.2} is approximately controllable on $J$.
\end{proof}


\section{An Example}
\label{sec:4}

In this section, we apply Theorems~\ref{Theorem:3.1} and \ref{Theorem:3.2}
to the following fractional partial functional differential inclusion
with nonlocal control condition:
\begin{equation}
\label{eq:4.1}
\frac{\partial^{\alpha}}{\partial t^{\alpha}}\left[u(x, t)-x\arctan u(x, \sin t)\right]
\in \frac{\partial^{2}u(x,t)}{\partial x^{2}}+\int_{0}^{t}b(t, s)\exp \xi\left(x, \sin s\right)ds,
\end{equation}
\begin{equation}
\label{eq:4.2}
u(x, 0)-u_{0}(x)=\sum\limits_{k=1}^{m}c_{k}\left[\xi(x, t_{k})-u(x, t_{k})\right],
\quad x\in [0,\pi],
\end{equation}
\begin{equation}
\label{eq:4.3}
u(0,t)=u(\pi,t)=0, \quad t\in J,
\end{equation}
where $0<\alpha\leq 1$, $0< t_{1}<\cdots<t_{m}< a$, $c_{k}$, $k=1,\ldots, m$,
are given constants and the function $b(t, s)$ is continuous on $\Delta$.
Let us take the function $g(t, u(\cdot))=x\arctan u(x, \cdot)$,
the multivalued map $f(t, s, \cdot)=b(t, s)e^{\cdot}$, the nonlocal function
given by $h(u(\cdot, t))=\sum_{k=1}^{m}c_{k}u(\cdot, t_{k})$,
the control functions $\mu_{1}(t)=\mu_{2}(t)=\xi(\cdot, t)$, where
$\xi: [0,\pi]\times J\rightarrow [0,\pi]$ is continuous, and the delays
$\sigma(t)=\delta(t)=\sin t$. Assume that $H=L^{2}[0, \pi]$ and
define $A: H\rightarrow H$ by $Aw=w^{\prime\prime}$ with domain
$$
D(A)=\left\lbrace w\in H: w, w^{\prime} \text{ are absolutely continuous, }
w^{\prime\prime}\in H, w(0)=w(\pi)=0\right\rbrace
$$
dense in the Hilbert space $H$. Then,
$$
Aw=\sum\limits_{n=1}^{\infty}n^{2}\langle w,w_{n}\rangle w_{n}, \quad w\in D(A),
$$
where $\langle\cdot, \cdot\rangle$ is the inner product in $L^{2}[0, \pi]$.
It is well known that $A$ generates a strongly continuous semigroup
$\lbrace Q(t)$, $t\geq 0\rbrace$ on $H$, which is compact,
analytic, and self-adjoint. Furthermore, $A$ has a discrete spectrum
with eigenvalues $-n^{2}, n\in \mathbb{N}$,
and the corresponding normalized eigenfunctions
are given by $w_{n}(t)=\sqrt{\frac{2}{\pi}}\sin nx$, $0\leq x\leq\pi$,
with $\lbrace w_{n}: n\in\mathbb{N}\rbrace$ an orthonormal basis of $H$ and
$$
Q(t)w
=\sum\limits_{n=1}^{\infty}e^{-n^{2}t}\langle w, w_{n}\rangle w_{n}
$$
for all $t\geq0$ and $w\in H$. In particular, $Q(\cdot)$ is a uniformly
stable semigroup and $\Vert Q(t)\Vert_{L^{2}[0, \pi]}\leq e^{-t}$.
Also, for each $w\in H$, $A^{-\frac{1}{2}}w=\sum\limits_{n=1}^{\infty}\frac{1}{n}\langle w$,
$w_{n}\rangle w_{n}$ with $\Vert A^{-\frac{1}{2}}\Vert_{L^{2}[0,\pi]}=1$
and the operator $A^{\frac{1}{2}}$ is given
on the space $D(A^{\frac{1}{2}})=H_{\frac{1}{2}} :=\left\lbrace w\in H:
\sum_{n=1}^{\infty}n\langle w, w_{n}\rangle w_{n}\in H\right\rbrace$ by
$A^{\frac{1}{2}}w=\sum\limits_{n=1}^{\infty}n\langle w, w_{n}\rangle w_{n}$.
Now define the infinite-dimensional space $Y$ by
$$
Y := \left\lbrace z=\sum\limits_{n=2}^{\infty}z_{n}w_{n}(x) \left\vert
\sum\limits_{n=2}^{\infty}\right.z_{n}<\infty\right\rbrace\subset L^{2}[0, \pi],
$$
where the norm in $Y$ is defined by $\Vert z\Vert=\sqrt{\sum_{n=2}^{\infty}z_{n}^{2}}$.
The control is defined by a bounded linear operator
$B=B_{1}=B_{2}: Y\rightarrow L^{2}[0, \pi]$
from the control Hilbert space $Y$ by
$$
(B\mu)(x)=2\mu_{2}w_{1}(x)+\sum\limits_{n=2}^{\infty}\mu_{n}w_{n}(x)
=\xi(x, t),\quad
\text{ for } \ z=\sum\limits_{n=2}^{\infty}\mu_{n}w_{n}\in Y.
$$
Assume that there exists a function $v(x, \delta(t))=\omega (x, \sin t)$  such that
$$
\omega(x, \sin t)+\xi(x, t) \in \int_{0}^{t} b(t, s)\exp \xi\left(x, \sin s\right)ds.
$$
Thus, the integral equation \eqref{eq:2.2} is satisfied.
Therefore, problem \eqref{eq:4.1}--\eqref{eq:4.3} is
an abstract formulation of the control system \eqref{eq:1.1}--\eqref{eq:1.2}.
Moreover, all the assumptions (H$_1$)--(H$_4$) hold. Then, the
associated linear system of \eqref{eq:4.1}--\eqref{eq:4.3} is not
exactly controllable but, by Theorems~\ref{Theorem:3.1} and \ref{Theorem:3.2},
the control system \eqref{eq:4.1}--\eqref{eq:4.3} is approximately controllable on $J$.
Note that Lemma \ref{Lemma:2.6} holds.


\section{Conclusion}
\label{sec:conc}

We have studied approximate controllability for a class of fractional delay dynamic inclusions.
We introduced, for the first time in the literature, a nonlocal control condition by setting
a control function that depends on the nonlocal condition, and other control that depends
on the multi-valued map that appears, as usual, in the right hand side of the inclusion.
Sufficient conditions for approximate controllability are obtained. In particular,
our conditions are formulated in such a way that approximate controllability
of the nonlinear dynamical system is implied by the approximate controllability of its corresponding
linear part. More precisely, the controllability problem is transformed into a fixed point problem
for an appropriate nonlinear operator in a suitable function space. Using fractional calculations,
multi-valued analysis, and Sadovskii's fixed point theorem, we guarantee the existence of a
fixed point of this operator and study approximate controllability of the considered systems.
Finally, an example is provided to illustrate the applicability of the new results.

In order to describe various real-world problems in physical and engineering sciences subject
to abrupt changes at certain instants during the evolution process, impulsive fractional differential
equations have become important in recent years as mathematical models of many phenomena
in both physical and social sciences \cite{AMA.44}. In \cite{AMA.11}, Debbouche and Baleanu
establish a controllability result for a class of fractional evolution nonlocal impulsive quasilinear
delay integro-differential systems in a Banach space by using the theory
of fractional calculus and fixed point techniques. Upon making some appropriate assumptions
on system functions and by adapting the techniques and ideas established here with those of
\cite{AMA.11}, one can prove approximate controllability of nonlocal control fractional
delay dynamic inclusions with impulses.

Degenerate differential equations of integer order are often used to describe various
processes in science and engineering \cite{AMA.54}. As an open problem for further investigations,
we mention the study of approximate controllability for degenerate fractional dynamic (stochastic) inclusions.


\section*{Acknowledgments}

This work was supported by Portuguese funds through the
\emph{Center for Research and Development in Mathematics and Applications} (CIDMA),
and \emph{The Portuguese Foundation for Science and Technology} (FCT),
within project PEst-OE/MAT/UI4106/2014. Debbouche was also supported
by FCT within the post-doc project BPD/UA/CIDMA/2011 --- PD2012-MTSC.
The hospitality and the excellent working conditions at the University
of Aveiro and CIDMA are gratefully acknowledged. The authors would
like also to thank the reviewers for their valuable comments.


\small



\end{document}